\documentclass[a4paper,10pt]{article}

\usepackage{amsmath}
\usepackage{amssymb}
\usepackage{amsthm}
\usepackage[latin1]{inputenc}
\usepackage{eurosym}

\usepackage[pdfpagelabels=false,colorlinks=true,linkcolor=blue,citecolor=red,urlcolor=cyan]{hyperref}

\usepackage{xcolor}
\usepackage{graphicx}
\usepackage{subfigure}
\usepackage{lipsum}

\newtheorem{theorem}{Theorem}
\numberwithin{theorem}{section}
\newtheorem{proposition}[theorem]{Proposition}

\newcommand{\gammanikj}[1]{\gamma^{n,#1}_{jl}}

\newcommand{\Deltak}[1]{\Delta_{#1}}
\newcommand{\un}[3]{u^{#1}_{n+e_{#2#3}}}
\newcommand{\vecun}{u_n}
\newcommand{\unone}{u_n^1}
\newcommand{\untwo}{u_n^2}

\DeclareMathOperator*{\argmin}{arg\,min}
\newcommand{\argminmu}{\argmin\limits_{\mu \in \mathbb{R}^2,~\mu \geq 0}}

\newcommand{\fractionsc}{\frac{n}{N}}
\newcommand{\fractiononetwo}{\frac{n+e_{12}}{N}}
\newcommand{\fractiontwoone}{\frac{n+e_{21}}{N}}

\newcommand{\Rr}{{\mathbb{R}}}
\newcommand{\Nn}{{\mathbb{N}}}

\newcommand{\Ii}{\mathcal I}

\newcommand{\Pp}{\mathcal P}

\newcommand{\bi}{{\bold{i}}}

\newcommand{\bk}{{\bold{k}}}
\newcommand{\bn}{{\bold{n}}}
\newcommand{\bom}{{\bold{m}}}

\begin{document}
\title{Socio-economic applications of finite state mean field games}
\author{Diogo Gomes, Roberto M. Velho, Marie-Therese Wolfram}

\date{}
\maketitle


\begin{abstract}
In this paper we present different applications of finite state mean field games to socio-economic sciences. Examples include paradigm shifts in the scientific community or the consumer choice behaviour in the free market. The corresponding finite state mean field game models are hyperbolic systems of partial differential equations, for which we present and validate different numerical methods. We illustrate the behaviour of solutions with various numerical experiments, which show interesting phenomena like shock formation. Hence we conclude with an investigation of the shock structure in the case of two-state problems.
\end{abstract}


\section{Introduction}

Mean field games have become a powerful mathematical tool to model the dynamics of agents in economics, finance and the social sciences. 
Different settings have been considered in the 
literature such as discrete and continuous in time or finite and continuous state space. Originally finite state mean field 
games, see \cite{LCDF, GMS, Gueant2, Gueant1,  GMS2, FG13}, were studied  as an attempt to understand the more general continuous state problems introduced by Lasry \& Lions in \cite{ll1,ll2, ll3} as well as by Huang et al. in \cite{Caines1, Caines2}.
For additional information see also the recent surveys such as \cite{llg2, cardaliaguet, achdou2013finite, GS, bens}. \\

\noindent In this paper we apply finite state mean field games to two
classical problems in socio-economic sciences: consumer choice behaviour and paradigm shifts in a scientific community, see \cite{BesancenotDogguy}.
These mean field game models give rise to systems of hyperbolic partial differential equations, for which 
we develop a new numerical method.
The mathematical modelling is based on the following situation. Let us consider a system of $N+1$ identical players or agents, which can switch between $d \in \mathbb{N}$ different states.  
Each player is in a state $i\in\Ii = \{1, \hdots, d\}$ and can choose a switching strategy to any other state $j\in\Ii$.
The only information available to each player, in addition to its own state, is the number $n_j$ of players he/she sees in the different states $j\in \Ii$.
The fraction of players in each state $i\in \Ii$ is denoted by $\theta_i=\frac{n_i}{N}$ and
we define the probability vector 
 $\theta = (\theta_1, \dots, \theta_d)$, $\theta\in \Pp(\Ii)=\{\theta\in \Rr^d\, :\, \sum_{i} \theta_i=1,\,
 \theta_i\geq 0\}$, which encodes the statistical information about the ensemble of players. 
Each player faces an optimisation problem over all possible switching strategies. Here the key question is
the existence of a Nash equilibrium and to determine the limit as the number of players tends to infinity.  
Gomes et al. showed in \cite{GMS2} that the $N+1$ player Nash equilibrium always exists.\\
\noindent Next we consider the limit $N\to\infty$. Gomes et al. proved in \cite{GMS2} that at least for short time  
the value $U^i=U^i(\theta,t)$ for a player in state $i$, when the distribution of players among the different states
is given by $\theta$, satisfies the hyperbolic system
\begin{equation}
\label{hsys}
\displaystyle - U_t^i(\theta,t) = \sum_{j\in \Ii} g_j(U, \theta) \ \partial_{\theta_j} U^i(\theta,t) + h(U,\theta,i), 
\qquad U(\theta, T)=\psi(\theta).  
\end{equation}
Here $U^i:\Pp(\Ii)\times [0,T]\to \Rr$, 
$g:\Rr^d\times\Pp(\Ii)\to \Rr^d$,  
$h:\Rr^d\times\Pp(\Ii) \times \Ii \to \Rr$, and $\psi:\Pp(\Ii)\to \Rr^d$, 
and $\partial_{\theta_j}$ denotes the partial derivative with respect to the variable $\theta_j$. \\

\noindent In general first order hyperbolic equations do not admit smooth solutions and one needs to consider an appropriate notion 
of solution. Adequate definitions of solutions are well known for conservation laws and equations which admit a maximum principle, e.g. Hamilton-Jacobi equations. 
Up to now the appropriate notion of solutions for \eqref{hsys}, which encodes the mean field limit, is not clear. 
In this paper we present a numerical method, which is based on the Nash equilibrium equations for $N+1$ agents. Therefore these equations will, 
 in case of convergence, automatically yield the appropriate limit. In particular, thanks to the results in \cite{GMS2}, convergence
always holds for short time. 

\noindent For certain classes of finite state mean field games, called potential mean field games, system \eqref{hsys} can be regarded as the gradient of a 
Hamilton-Jacobi equation, see \cite{LCDF, GMS2}. We use this property to validate the presented numerical method. Our computational experiments show the expected formation of shocks. We analyse these shock structures,
by introducing an auxiliary conservation law. This allows us to derive a Rankine-Hugoniot condition that characterises the qualitative behaviour of such systems. \\

\noindent This paper is organised as follows: we start section \ref{smfg} by recalling the $N+1$ player model, which gives rise
both to \eqref{hsys} and to the numerical method presented here.
Section \ref{ses} focuses on two-state mean field games, the numerical implementation of \eqref{hsys}
and our applications in socio-economic sciences.  We illustrate the behaviour of the various model
in section \ref{s:numericalexamples}. In section \ref{sstsp} we briefly analyse  the shock structure for two-state mean field games.

\section{Finite state mean field games}
\label{smfg}

We start with a more detailed presentation of finite state mean field games and the formal derivation of \eqref{hsys}.

\subsection{$N+1$ player problem}

We consider a system of $N+1$ identical players or agents. We fix one of them, called the {\em reference player}, and
denote by $\bi^t$ its state at time $t$. All other players can be in any state $j\in \Ii$ at each time $t$. We denote by 
$\bn_j^t$ the number of
players (distinct from the reference player) at state $j$ at time $t$, and $\bn^t=(\bn_1^t, \hdots,  \bn_d^t)$.   
Players change their states according to two mechanisms: either by a 
Markovian switching rate chosen by the player, or by interactions with the other players.
We assume for the moment that all players (except the reference player) have
chosen an identical Markovian strategy $\beta(n,i,t)\in (\Rr^+_0)^d$. 
The reference player is allowed to choose a possibly different strategy 
 $\alpha(n,i,t)\in (\Rr^+_0)^d$. Given these strategies, the joint process $(\bi^t, \bn^t)$ is a time
 inhomogeneous Markov chain (neither $\bn^t$ nor $\bi^t$ are Markovian when considered separately).
 The generator of this process can be written as
\[
A\varphi^i_n=A^\alpha_0 \varphi^i_n+
A^{\beta}_1\varphi^i_n+A^\omega_2 \varphi^i_n,  
\]
where the definition of each term will be given in what follows. Let
$e_k$ be the $k-th$ vector of the canonical basis of $\Rr^d$ and  $e_{jk}=e_j-e_k$. Then 
\[
A^\alpha_0 \varphi^i_n=
\sum_{j\in \Ii} 
\alpha_j \ \left(\varphi^j_{n}-\varphi^i_n \right),
\]
\[
\displaystyle A^{ \beta}_1\varphi^i_n=
\sum_{j, k\in \Ii}\gamma_{\beta,jk}^{n,i} \left(\varphi^i_{n+e_{jk}}-\varphi^i_n \right),
\]
and
\begin{equation*}
\resizebox{1\hsize}{!}{$
\displaystyle A^\omega_2 \varphi^i_n=
\sum_{j, k\in \Ii}
\frac{\omega_{jk} n_j n_k}{N^2}
\left(
\varphi^i_{n+e_{jk}}+\varphi^i_{n+e_{kj}}-2\varphi^i_n
\right)
+
\sum_{j\in \Ii}
\displaystyle \frac{\omega_{ij} n_j}{N^2}
\left(
\varphi^j_{n}+\varphi^i_{n+e_{ij}}-2\varphi^i_n
\right).
$
}
\end{equation*}
The terms $A^\alpha_0$ and $A^\beta_1$ correspond to the transitions due to the switching strategies $\alpha$ and $\beta$.
Select one of the players distinct from the reference player. Denote by $\bk^t$ its position at time $t$,
and call $\bom^t$ the vector $\bom^t=\bn^t+e_{\bi^t\bk^t}$, the process that records the number of other players in any state from the point of view of this player. 
Suppose further that there are no interactions ($A_2^\omega=0$). Then for $j\neq k$, we have
$$
\mathbb{P}\Big(\bk^{t+\delta}=j\|\bom^t=m,\bk^t=k\Big)= \beta_{j}(m,  k, t)  \ \delta+o(\delta)\,.
$$
Assuming symmetry and independence of transitions 
from any state $k$ to a state $j$, $k \neq j$, we have
\begin{align*}
\mathbb{P}\Big(\bn^{t+\delta}=n+e_{jk}\|\bn^t&=n,\bi^t=i\Big)= \gamma_{\beta,jk}^{n,i}(t)  \ \delta+o(\delta)\,, 
\end{align*}
where  the transition rates of the process $\bn^t$ are given by
\begin{align}\label{gamma}
\gamma_{\beta,jk}^{n,i}(t)&= n_k  \ \beta_{j}(n+e_{ik},k, t).
\end{align}
Similarly, the reference player switching probabilities are
\[
\mathbb{P}\Big(\bi^{t+\delta}=j|\bi^t=i, \bn^t=n\Big)=\alpha_j(n, i, t)  \ \delta +o(\delta). 
\]
The transitions between different states due to interactions  give rise to the
term $A^\omega_2$. Its particular structure comes from the assumption that 
any two players with distinct states $j$ and $k$
can meet with rate $\frac{\omega_{jk}}{N}$ (with $\omega_{kj}=\omega_{jk}\geq 0$). As a result of this interaction, 
either both end in state $j$ or state $k$ (with probability $\frac 1 2$ respectively).

\noindent We assume that all players have the same running cost determined by a function
$c:\Ii\times \Pp(\Ii)\times (\Rr_0^+)^d\to \Rr$ as well as an identical terminal cost $\psi(\theta)$, which is
Lipschitz continuous in $\theta$.
The running cost $c(i, \theta, \alpha)$ depends on the state $i$ of the player, the mean field $\theta$, that is the distribution of players among states, and on 
the switching rate $\alpha$. 
As in \cite{GMS2}, we suppose that $c$ is 
Lipschitz continuous in $\theta$ with a Lipschitz constant (with respect to $\theta$) bounded independently of $\alpha$.
Let the running cost $c$ be differentiable with respect to $\alpha$, and $\frac{\partial c}{\partial \alpha}(i, \theta, \alpha)$
be Lipschitz with respect to $\theta$, uniformly in $\alpha$. We assume that for each $i\in \Ii$, 
the running cost $c(i,\theta,\alpha)$ does not depend on the $i$-{th} coordinate $\alpha_{i}$ of $\alpha$. Furthermore we make
the additional assumptions on $c$:
\begin{enumerate}
\item[(A1)] For any $i\in \Ii$, $\theta\in \Pp(\Ii)$, $\alpha,  \alpha'\in (\Rr_0^+)^d$, with $\alpha_j \neq  {\alpha}_{j}'$, for
some $j\neq i$,
\begin{equation}
\label{convc}
c(i,\theta,\alpha\,')-c(i,\theta,\alpha)\geq \nabla_{\alpha}\  c(i,\theta,\alpha)\cdot(\alpha\,'-\alpha)+\gamma\|\alpha\,'-\alpha\|^2. 
\end{equation}
\item[(A2)] The function $c$ is superlinear on $\alpha_{j}$, $j \neq i$, that is,
\[
\lim_{\alpha_j\to\infty} \frac{c(i, \theta,  \alpha)}{\|\alpha\|}\to \infty.
\]
\end{enumerate}

\noindent Let us fix a reference player and set the Markovian strategy $\beta$ for the remaining $N$ players. 
The objective of the reference player is 
to minimise its total cost, whose minimum 
over all Markovian strategies $\alpha$
is given by
\begin{equation}\label{defu} 
u^{i,\beta}_n(t)= \inf_{\alpha}\mathbb{E}^{\beta,\alpha}_{(\bi_t, \bn_t)=(i, n)} \left[ \int_t^T c\left(\bi_s,\frac{\bn_s}N,\alpha(s)\right)ds + \psi^{\bi_T}
\left(\frac{\bn_T}N\right) \right]\;.
\end{equation}

\noindent We define $ \displaystyle \Delta_i \varphi_n (t) = \left(\varphi_n^1(t)-\varphi_n^i(t), \hdots , \varphi_n^d(t)-\varphi_n^i(t)\right)$.
The generalised Legendre transform of $c$ is given by
\begin{equation}
\label{hami}
h(z, \theta, i)=\min_{\mu\in {(\Rr_0^+)^d}} c(i, \theta, \mu) +\mu\cdot \Delta_i z.
\end{equation}
Note that $h$ only depends on the differences between coordinates of the variable $z$, that is, if
$\Delta_i z=\Delta_i\tilde z$ then $h(z, \theta, i)=h(\tilde z, \theta, i)$.

\noindent
The function $u_n^{i,\beta}$ is the solution to the ODE
\[
\displaystyle -\frac{\partial u_n^{i,\beta}}{\partial t}
=h \left(u_n^{i,\beta}, \frac{n}{N}, i\right)+
A^\beta_1 u_n^{i,\beta}+A^\omega_2 u_n^{i,\beta}. 
\]
Next we define, for $j\neq i$, 
\begin{align}
\alpha^*_j(z,\theta, i)=\argmin_{\mu\in {(\Rr_0^+)^d}}c(i, \theta, \mu) +\mu\cdot \Delta_i z. 
\label{alpha_expression} 
\end{align}
If $h$ is 
differentiable, 
for $j \neq i$,
\begin{equation}
\label{otheralpha_expression}
\alpha_j^*(\Delta_iz,\theta,i) = \frac{\partial h\left(\Delta_iz,\theta,i\right)}{\partial z^j}.
\end{equation}
For convenience and consistency with \eqref{otheralpha_expression}, we require
\begin{equation}
\label{sumalf}
\sum_{j\in \Ii}\alpha^*_j(z,\theta, i)=0.
\end{equation}
Then the optimal strategy for the reference player is given by
\[
\bar \alpha(n, i, t)=\alpha^*\left(\Delta u_n^i,\frac n N, i\right). 
\]
We say that a strategy $\beta$ is a Nash equilibrium if
the optimal response of the reference player is $\beta$ itself, i.e., $\beta=\bar \alpha$. 
Thus setting $u_n^i=u_n^{i,\bar \alpha}$ we have the Nash 
equilibrium equation for the value function 
\begin{align}
\label{npu}
&-\frac{\partial u_n^i}{\partial t}
=h\left(u_n^i, \frac{n}{N}, i\right)+
\sum_{j, k\in\Ii}\gamma_{jk}^{n,i} \left(u^i_{n+e_{jk}}-u^i_n\right)\\\notag
&\quad +\sum_{j, k\in\Ii}
\frac{\omega_{jk} n_j n_k}{N^2}
\left(u^i_{n+e_{jk}}+u^i_{n+e_{kj}}-2u^i_n
\right)
+
\sum_{j\in\Ii}
\frac{\omega_{ij} n_j}{N^2}
\left(
u^j_{n}+u^i_{n+e_{ij}}-2u^i_n
\right), 
\end{align}
where we define
\begin{equation}
\label{gamadef}
\gamma_{jk}^{n,i}(t)= n_k \bar \alpha_{j}\left(n+e_{ik},k, t\right).
\end{equation}

\subsection{Formal asymptotic behaviour}

Now we investigate the asymptotic behaviour as $N\to \infty$ of the $N+1$ player dynamics \eqref{npu}. 
For that we suppose there is a smooth function $U:\Pp(\Ii)\times[0,T]\to \Rr^d$ such that 
\[
u_n^i(t)=U^i\left(\frac n N, t\right).
\]
Then we have the following expansions:
\begin{equation*}
\resizebox{.985\hsize}{!}{$
A^{\bar \alpha}_1u_n^i=\sum\limits_{j, k\in\Ii}
\theta_k \left[ \left(1+\frac{\partial_{\theta_i}-\partial_{\theta_k} }{N} \right)\alpha^*_j(\Delta_k U, \theta, k)\right]
\left[ \left( (\partial_{\theta_j}-\partial_{\theta_k} ) +\frac{(\partial_{\theta_j}-\partial_{\theta_k})^2}{2N}\right) U^i\right]+O\left(\frac 1 {N^2} \right),
$} 
\end{equation*}
\begin{align*}
A^\omega_2 u_n^i=&\sum_{j, k\in\Ii}
\frac{\omega_{jk}}{N} \theta_j\theta_k \left(\partial^2_{\theta_j \theta_j} U^i+\partial^2_{\theta_k \theta_k} U^i -2\partial^2_{\theta_j \theta_k} U^i\right)
+\sum_{j\in\Ii}\omega_{ij}(U^j-U^i) \\
&
+\sum_{j\in\Ii} \omega_{ij} \ \theta_j \left(\partial_{\theta_i} U^i-\partial_{\theta_j} U^i \right)+O\left(\frac{1}{N^2}\right).
\end{align*}
We observe that for fixed $j$ and $k$ the operators
$
\frac{\theta_k \ \alpha^*_j}{2} \left(\partial_{\theta_j}-\partial_{\theta_k} \right)^2$, and \linebreak $\displaystyle \omega_{jk} \ \theta_j\theta_k \left(\partial_{\theta_j}-\partial_{\theta_k}\right)^2
$ 
are degenerate elliptic operators. The first one is degenerate because $\theta_k, \alpha^*_j\geq 0$; the second because $\omega_{jk}\geq 0$. 
Hence their sum is also a degenerate operator. Therefore the combination of the second order terms in the expansion of $A^{\bar \alpha}_1$ and $A^\omega_2$
can be written as
\[
 \sum\limits_{l,m\in\Ii} b_{lm} \ \partial^2_{\theta_l \theta_m}=
\sum\limits_{j,k\in\Ii}
\frac{\theta_k \ \alpha^*_j+2 \omega_{jk} \ \theta_j \theta_k}{2} \ \left(\partial_{\theta_j}-\partial_{\theta_k}\right)^2, 
\]
for a suitable non-negative matrix $b$. We conclude that \eqref{npu} can be formally  approximated 
by the parabolic system
\begin{equation}
\label{psys}
- \ U_t^i(\theta,t) = \sum_{j\in\Ii} g_j^N(U, \partial_\theta U,\theta,  i) \ \partial_{\theta_j} U^i + h(U, \theta,i) + \frac 1 N \sum_{l,m \in\Ii}\ b_{lm}(U, \theta)  \ \partial_{\theta_l \theta_m}^2 U^i, 
\end{equation}
with suitable $g^N:\Rr^d\times\Rr^d\times\Pp(\Ii)\times \Ii\to \Rr^d$.
Furthermore $g^N$ converges locally uniformly in compacts to
\begin{equation}
\label{formforg}
g_j(U,\theta)=\sum_{i\in\Ii} \theta_i\alpha^*_j(U, \theta, i).
\end{equation}
This implies that the limit of \eqref{psys} is \eqref{hsys}, which does not depend on the interaction between players~$(\omega_{jk})$. 
Note that 
\begin{equation}
\label{symmetrygj}
\sum_{j \in \Ii } g_j(U,\theta) = \sum_{j \in \Ii} \ \sum_{i \in \Ii } \theta_i \alpha_j^*(U,\theta,i)
= \sum_{i \in \Ii} \theta_i \ \sum_{j \in \Ii} \alpha_j^*(U,\theta,i) =  0,
\end{equation}
since $\sum\limits_{j \in \Ii } \alpha_j^*(U,\theta,i) =  0$, from \eqref{sumalf}. 
Additionally, 
\begin{equation}
\label{propg}
g_j(U, \theta, i)=g_j(\Delta_i U, \theta, i),
\end{equation}
using in \eqref{formforg} the equation \eqref{alpha_expression}.

\subsection{Potential mean field games}
\label{secpmfg}

Next we consider a special class of mean field games, in which system \eqref{hsys} can be written as
the gradient of a Hamilton-Jacobi equation. Suppose that
\begin{equation}
\label{seph}
h(u, \theta, i)=\tilde h(u, i)+f(i, \theta), \ \ \ i \in \Ii,   
\end{equation}
and $f(i, \theta)=\partial_{\theta_i} F(\theta)$, for some potential $F:\Rr^d\to \Rr$. We set
\begin{equation}
\label{capH}
H(u, \theta)=\sum_{k \in \Ii} \theta_k \ \tilde h(\Delta_k u, k) +F(\theta). 
\end{equation}
Let $\Psi_0: \mathbb{R}^d \rightarrow \mathbb{R}$ be a continuous function and consider 
a smooth enough solution
$\Psi: \mathbb{R}^d \times[0,T] \rightarrow \mathbb{R}$ to the Hamilton-Jacobi equation
\begin{equation}
\label{e:hj}
\begin{cases}
 \displaystyle -\frac{\partial \Psi(\theta,t)}{\partial t} = H \left(\partial_{\theta} \Psi, \theta \right), & \vspace{0.15cm} \\ 
\Psi(\theta, T) = \Psi_0(\theta).&
\end{cases}
\end{equation}
Note that $\theta\in \Rr^d$. In some cases it is possible to reduce the dimensionality at the price of introducing suitable boundary conditions (see section \ref{s:numericalexamples}-(\ref{s:numpotmfg})). This reduction will be used in the applications presented later.

\noindent Set $U^j(\theta, t)=\partial_{\theta_j} \Psi(\theta, t)$. If we differentiate \eqref{e:hj} with respect to $\theta_i$ we obtain
\begin{align*}
 -U^i_t = \sum_{j \in \Ii} \partial_{u_j} H(U,\theta) ~ \partial_{\theta_i} U^j + \tilde h(\Delta_i U, i) + \partial_{\theta_i} F.
\end{align*}
The first  term on the right hand side can be written as 
\begin{align*}
\sum_{j \in \Ii}  \partial_{u_j} H(U,\theta) \ \partial_{\theta_i} U^j  =
\sum_{k,j \in \Ii} \theta_k \ \partial_{u_j} \tilde h(\Delta_k U, k) \  \partial_{\theta_i} U^j
=\sum_{j \in \Ii} g_j (U, \theta) \ \partial_{\theta_j} U^i, 
\end{align*}
taking into account the identity $\partial_{\theta_i} U^j=\partial_{\theta_j} U^i$.
From this we get that
\[
 -U^i_t = \sum_{j \in \Ii} g_j (U, \theta) \ \partial_{\theta_j} U^i + \tilde h(\Delta_i U, i) + \partial_{\theta_i} F,
\]
and deduce, using \eqref{seph}, that $U^i$ is indeed a solution of \eqref{hsys}. 

\noindent Potential mean field games have remarkable properties and connections to calculus of variations. For instance, long time convergence properties
 of these problems can be addressed through $\Gamma$-convergence techniques, see for instance \cite{FG13}.

\section{Two-state mean field games in socio-economic sciences}\label{ses}
 
In this section we present several applications of
finite state mean field games to socio-economic sciences. In order to keep the presentation simple
we consider only two-state problems.   
We start by stating the explicit equations, where agents can choose between two options. 
The classical examples discussed here are the consumer choice behaviour and a paradigm shift model in the scientific 
community. Note that the number of choices can be increased, but we focus on two-state applications for reasons of clarity
and readability.

\subsection{Two-state problems}
\label{tsp}

\noindent Consider a two-state mean field game, where the fraction of players in either state, $1$ or $2$, is given by $\theta_i$, $i=1,2$ with 
$\theta_1+\theta_2 = 1$, and $\theta_i \geq 0$. 
Since the limit equation \eqref{hsys} does not depend on the interactions (although the $N+1$ player model does), 
we set $\omega=0$. Note that $\omega\neq 0$ would result in different numerical methods (and potentially different solutions) for \eqref{hsys}.
We suppose further that 
the running cost $c = c(i,\theta,\mu)$ in \eqref{defu} depends quadratically on the switching rate $\mu$, i.e.
\begin{align}\label{e:cost}
c(i,\theta,\mu) = f(i,\theta) + c_0(i,\mu), \text{ with } c_0(i,\mu) = \frac{1}{2} \sum_{j \neq i}^2 \mu_j^2.
\end{align}
Then
\begin{align}
\label{e:h}
\hspace{-0.18cm}
h(z,\theta,1) = f(1,\theta) - \frac{1}{2} \left((z^1-z^2)^+ \right)^2 \text{ and }
 h(z,\theta,2) = f(2,\theta) - \frac{1}{2} \left((z^2-z^1)^+ \right)^2.
\end{align}
The optimal switching rate $\alpha^*$ is given by:
\begin{align*}
\alpha^*(z,\theta,1) = \argminmu \bigl[ f(1,\theta) + \frac{1}{2} \mu_2^2 + \left(\begin{smallmatrix}\mu_1\\\mu_2\end{smallmatrix}\right) \cdot \left(\begin{smallmatrix} 0 \\ z^2-z^1 \end{smallmatrix}\right)
\bigr] \Rightarrow \alpha_2^*(z,\theta,1) = (z^1-z^2)^+,
\end{align*}
\begin{align*}
\alpha^*(z,\theta,2) = \argminmu \left[ f(2,\theta) + \frac{1}{2} \mu_1^2 + \left(\begin{smallmatrix}\mu_1\\\mu_2\end{smallmatrix}\right) \cdot \left(\begin{smallmatrix} z^1-z^2 \\ 0 \end{smallmatrix}\right) 
\right] \Rightarrow \alpha_1^*(z, \theta,2) = (z^2-z^1 )^+.
\end{align*}
Since
\begin{align*}
&\alpha^*_1(U,\theta,1) = - \left(U^1 - U^2\right)^+; \hspace{1cm} \alpha^*_2(U,\theta,1) = \ \ \left(U^1 - U^2\right)^+; \\ 
&\alpha^*_1(U,\theta,2) = \ \ \left(U^2 - U^1\right)^+; \hspace{1cm} \alpha^*_2(U,\theta,2) = - \left(U^2 - U^1\right)^+,
\end{align*}
we conclude from \eqref{formforg} that
\begin{align}
g_1(U,\theta) &= -\theta_1  \left(U^1 - U^2\right)^+ + \theta_2 \ \left(U^2 - U^1\right)^+, \label{formg1}
\\
g_2(U,\theta) &= \ \theta_1 \ \left(U^1 - U^2\right)^+ - \theta_2 \ \left(U^2 - U^1\right)^+ =- \ g_1(U, \theta). \nonumber
\end{align}
\noindent Note that if the function $f$ is a gradient field, i.e. $f=\nabla F$, the two-state problem is a potential mean field game, cf. section \ref{smfg}-(\ref{secpmfg}). 
In this case, for $p=(p_1, p_2)\in \Rr^2$, \eqref{capH} is given by 
\begin{align}\label{e:Hamiltonian}
H(p, \theta)=F(\theta)-\frac{\theta_1 \left((p_1-p_2)^+\right)^2+\theta_2 \left((p_2-p_1)^+\right)^2}{2}. 
\end{align}

\noindent The above calculations allow us to introduce a numerical method  based on the two-state mean field model for $N+1$ players. Let $n_i$, $i=1,2$ denote
the number of players in state $i$ (as seen by the reference player excluding itself) and $N = n_1+n_2$.
The vector $n$ gives the number of players in each state, i.e. $n = (n_1, n_2) = (n_1, N-n_1)$. 
As in \eqref{gamadef}  we have
\begin{align*}
\gamma^{n,1}_{jl} = n_l \ \alpha^*_j \left(\Deltak{l} \un{}{1}{l},\frac{n+e_{1l}}{N},l\right) \text{ and } \gamma^{n,2}_{jl} = n_l \ \alpha^*_j \left(\Deltak{l} \un{}{2}{l},\frac{n+e_{2l}}{N},l \right).
\end{align*}

\noindent Then equation \eqref{npu} for the value function $u^i_n$ reads as
\begin{align}\label{e:ui}
\left\{
\begin{array}{l}
\displaystyle - \frac{d\unone}{dt} = \sum\limits_{j,l=1}^2 \gammanikj{1} \left(\un{1}{j}{l} - \unone\right) + h\left(\Deltak{1} \vecun,\fractionsc,1\right), \vspace{0.10cm} \\ 
\displaystyle - \frac{d\untwo}{dt} = \sum\limits_{j,l=1}^2 \gammanikj{2} \left(\un{2}{j}{l} - \untwo\right) + h\left(\Deltak{2} \vecun,\fractionsc,2\right),
\end{array}
\right. 
\end{align}
which can be rewritten as
\begin{equation}
\label{e:discretescheme}
\hspace{-0.048cm}
\begin{cases}
\displaystyle -\frac{d \unone}{dt} =& (N-n_1) \ \alpha^*_1 \left(\Deltak{2} \un{}{1}{2}, \fractiononetwo, 2\right) \left(\un{1}{1}{2} - \unone \right) \vspace{0.12cm}  \\
 &+ \ n_1 \ \alpha^*_2 \left(\Deltak{1} \vecun, \fractionsc, 1 \right) \left(\un{1}{2}{1} - \unone \right) +  h \left(\Deltak{1} \vecun, \fractionsc,1 \right), \vspace{0.20cm} \\
\displaystyle -\frac{d \untwo}{dt} =& (N-n_1) \ \alpha^*_1 \left(\Deltak{2} \vecun, \fractionsc, 2 \right) \left(\un{2}{1}{2} - \untwo \right) \vspace{0.12cm}  \\
&+ \ n_1 \ \alpha^*_2 \left(\Deltak{1} \un{}{2}{1}, \fractiontwoone, 1 \right) \left(\un{2}{2}{1} - \untwo \right) +  h \left(\Deltak{2} \vecun, \fractionsc,2 \right). 
\end{cases}
\end{equation}

\noindent Since $\theta_1+\theta_2=1$ we use $\theta=(\zeta, 1-\zeta)$, for $\zeta\in [0,1]$. We split the domain $[0,1]$ into $N$ equidistant subintervals and define $\zeta_k = \frac{k}{N}$, $0\leq k\leq N, \ k\in \Nn$. The variable $\zeta_k$ corresponds to the fraction of players in state $1$. Then the fraction of players in state $2$ is given by $1-\zeta_k = \frac{N-k}{N}$.
Consequently \eqref{e:discretescheme} takes the form 
\begin{equation}
\label{e:discrete}
\resizebox{0.913\hsize}{!}{$
\begin{cases}
\displaystyle -\frac{d u_k^1}{dt} &= N(1-\zeta_k) \left(u^2_{k+1} - u^1_{k+1}\right)^+ \left(u_{k+1}^1-u_k^1\right)\vspace{0.12cm} \\
& + N \zeta_k \left(u_k^1 - u_k^2\right)^+\left(u_{k-1}^1 - u_k^1\right) + f(1,\zeta_k) - \frac{1}{2} \left(\left(u_k^1-u_k^2\right)^+\right)^2,\vspace{0.20cm}  \\
\displaystyle -\frac{d u_k^2}{dt} &= N(1-\zeta_k) \left(u^2_k - u^1_k\right)^+ \left(u_{k+1}^2-u_k^2\right) \vspace{0.12cm}\\
& + N \zeta_k \left(u_{k-1}^1 - u_{k-1}^2\right)^+\left(u_{k-1}^2 - u_k^2\right) + f(2,1-\zeta_k) - \frac{1}{2} \left(\left(u_k^2-u_k^1\right)^+\right)^2.
\end{cases}
$}
\end{equation}
Note that in system \eqref{e:discrete} the terms $\zeta_k$ and $(1-\zeta_k)$ vanish for $k=0$ and $k=N$, thus
no particular care has to be taken concerning the ghost points at $\zeta_{N+1}$ and $\zeta_{-1}$. This is the discrete analogue to 
not imposing boundary conditions on \eqref{hsys}. A similar situation occurs in 
state constrained problems for Hamilton-Jacobi equations. 

\subsection{Paradigm shift}\label{s:paradigm}

\noindent According to Kuhn, see \cite{Kuhn}, a paradigm shift corresponds to a change in a basic assumption within the ruling theory of science. Classical cases of
paradigm shifts are the transition from Ptolemaic cosmology to Copernican one, the development of quantum mechanics which replaced
classical mechanics on the microscopic scale or the acceptance of Mendelian inheritance as opposed to Pangenesis. Bensancenot and Dogguy modelled
a paradigm shift in a scientific community by a two-state mean field game approach and analysed the competition between two different scientific
hypothesis, see \cite{BesancenotDogguy}. In our example we consider a simpler model, but follow their general ideas and assumptions. \\
\noindent Let us consider a scientific community with $N$ researchers working on two different hypothesis. Each researcher working on
paradigm $i$, $i=1,2$, wants to maximise his/her productivity measured by a cost function of the form \eqref{e:cost}.
Here the function $f = f(i, \theta)$ corresponds to the productivity of a researcher working on paradigm $i$, and
 $ c_0(i,\mu) = -\frac{1}{2} \sum_{j \neq i}^2 \mu_j^2$ to the cost of switching to the other objective. Note the negative sign of the switching costs,
since agents want to maximise their productivity. We assume that the productivity is directly related to the number of 
researchers working on the paradigm, since for example more scientific activities like conference and collaborations.  In the
case of two different fields, $\theta_1$ gives the fraction of researchers working on paradigm $1$ and $\theta_2 = 1-\theta_1$ on paradigm $2$. We choose 
the functions $f(i,\theta)$, $i=1,2$, of the form
\begin{align}\label{e:paradigm}
& f(1,\theta) = \left[a_1 \ {\theta_1}^r + (1-a_1) \ (1-\theta_1)^r\right]^\frac{1}{r},\\ 
& f(2,\theta) = \left[a_2 \ (1-\theta_2)^r + (1-a_2) \ {\theta_2}^r\right]^\frac{1}{r} \nonumber.
\end{align} 
These functions are called {\em productivity functions with constant elasticity of substitution} and are commonly used in economics to combine 
two or more productive inputs (in our case scientific activities in the different fields) to an output quantity. The constant $r \in \mathbb{R}, r \neq 0$,
denotes the {\em elasticity of substitution}, and it measures how easy one can substitute one input for the other. The constants $a_i \in [0,1]$ measure the dependence of paradigm $i$ with respect to the other. If $a_i$ is close to one,
the field is more autonomous and little influenced by the activity in the other field. 

\subsection{Consumer choice}\label{s:mobilephone}
\noindent Consumer choice models relate preferences to consumption expenditure. We consider two choices of consumption goods and denote by $\theta_1$ the
fraction of agents consuming good $1$ and by $\theta_2 = 1-\theta_1$ the fraction consuming good $2$. We assume that the price of a good is strongly
determined by the consumption rate, in particular we choose, for $i=1,2$,  
\begin{align}\label{fcc}
f(i,\theta) &= 
\begin{cases} \displaystyle \frac{\theta_i^{1-\eta}-1}{1-\eta} + s_i,  &\eta > 0, \eta \neq 1, \vspace{0.12cm} \\
\ln(\theta_i)+ s_i, &\eta = 1,
\end{cases} 
\end{align}
where $s_i \in \mathbb{R}^+$ corresponds to the minimum price of the good. In economic literature the function $f$ is called the {\em isoelastic utility function}. 
 
\section{Numerical simulations}\label{s:numericalexamples}

We illustrate the behaviour of the discrete system \eqref{e:discrete} with several examples. Let $N=100$, i.e.
the interval $[0,1]$ is discretized into $100$ equidistant intervals.
Each grid point corresponds to the the percentage
of players being in state 1. System \eqref{e:discrete} is solved using an explicit in time discretization with
time steps of size $\Delta t = 10^{-4}$. In all examples in this section
the terminal time $T$ is set to $T=10$ if not stated otherwise.

\subsection{Numerical examples}

\paragraph{Example I (Shock formation):}In this first example we would like to illustrate the formation of shocks, a phenomena well known for
Hamilton-Jacobi equations. We choose a terminal cost of the form
\begin{align*}
u^1(\theta,10) = \theta_1-\frac{1}{2} \text{  and  } u^2(\theta, 10) = \theta_2-\frac{1}{2}, 
\end{align*}
a running cost as in \eqref{e:cost} with $f(1,\theta) = 1-\theta_1$ and $f(2,\theta) = 1-\theta_2 = \theta_1$.
Figure~\ref{f:ex1} clearly
illustrates the formation of a shock for smooth terminal data. This shock is also evident when we
consider the difference $u^1-u^2$ of the utilities. This difference is a relevant variable in this problem,
since both $g$ and $h$, given by \eqref{propg} and \eqref{hami} respectively, depend only on the difference 
between the utilities. This structure will be explored in more detail in section \ref{sstsp}.

\begin{figure}[htb!]
\begin{center}
\subfigure[Utility functions $u^1$ and $u^2$.]{\includegraphics[width=0.475\textwidth]{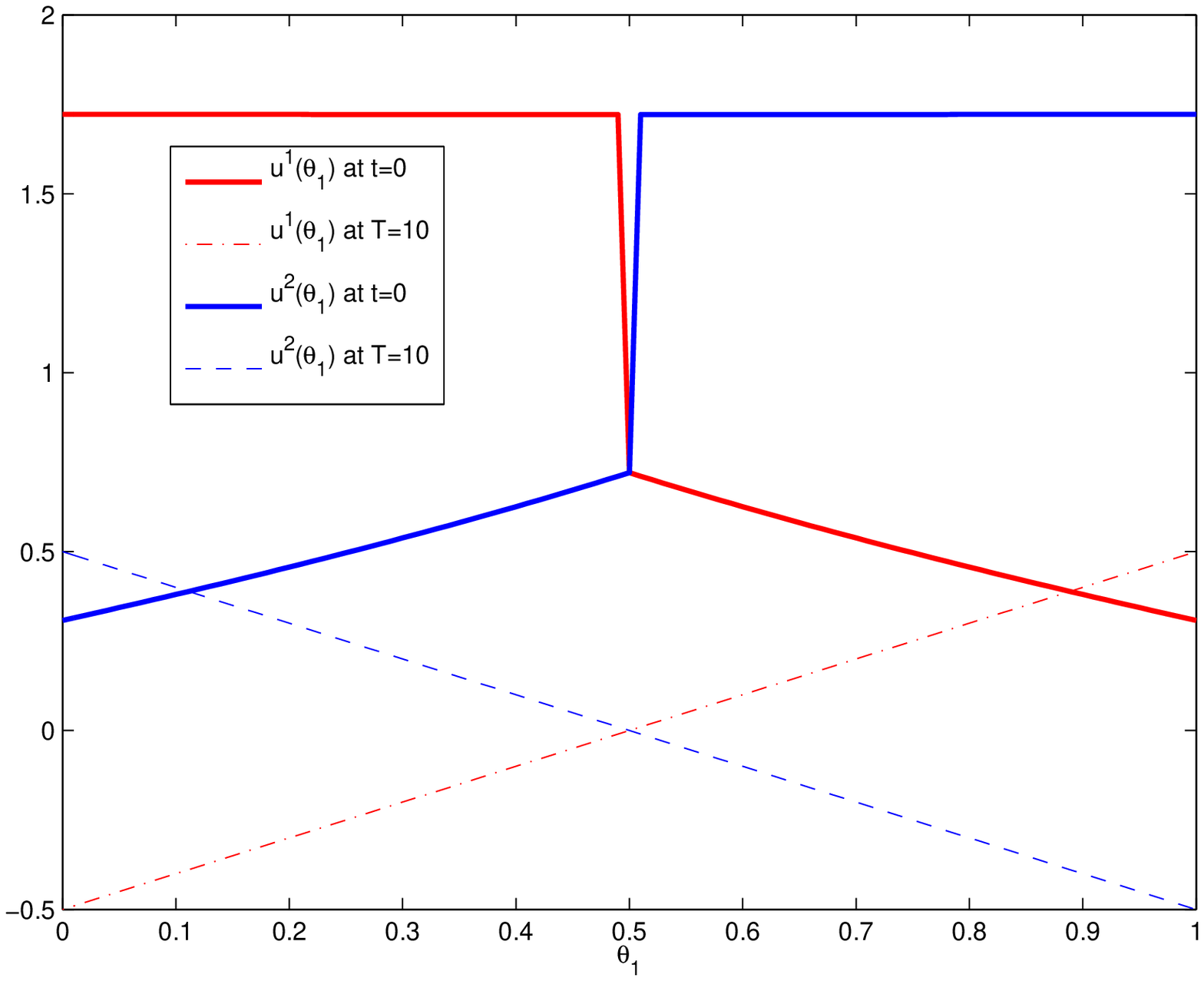}}\hspace*{0.5cm}
\subfigure[Difference $u^1-u^2$ at different times. ]{\includegraphics[width=0.475\textwidth]{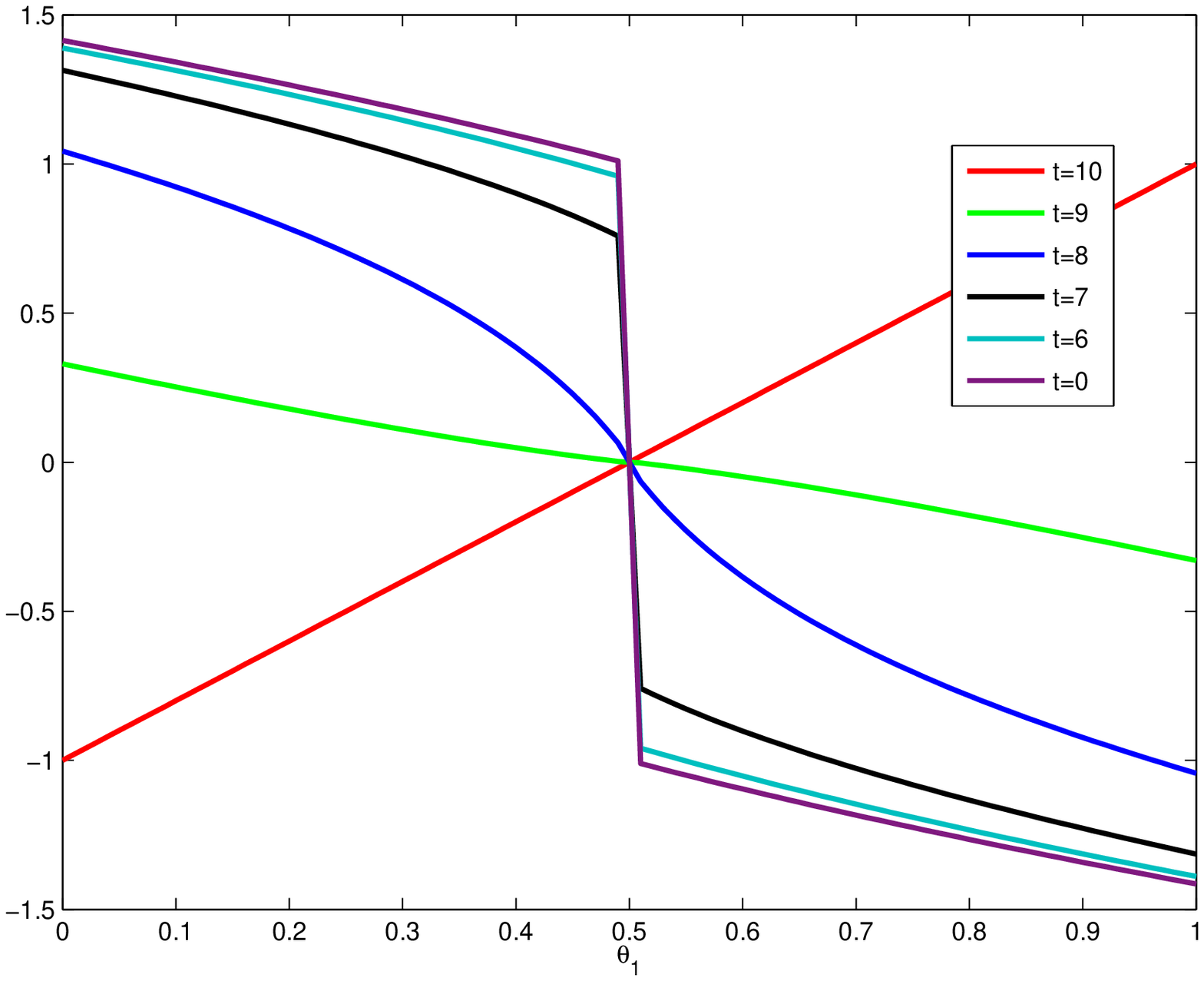}}
\caption{Example I - Shock formation.}\label{f:ex1}
\end{center}
\end{figure}
\paragraph{Example II (Paradigm shift):} In this example we illustrate the outcome of a two-state mean field game modelling a paradigm shift (section \ref{ses}-(\ref{s:paradigm})) within a scientific community. Note that we use the negative cost functional in this example, since we always consider minimisation problems. The terminal utilities are given by
\begin{align}\label{e:finalu}
u^1(\theta, T=10) = 1-\theta_1 \text{ and } u^2(\theta, T =10) = \theta_2, 
\end{align}
and the parameters in \eqref{e:paradigm} are set to
$a_1 = \frac{1}{2}, ~a_2 = \frac{9}{10},~r = \frac{3}{4}$.
\begin{figure}[htb!]
\begin{center}
\includegraphics[width=0.6\textwidth]{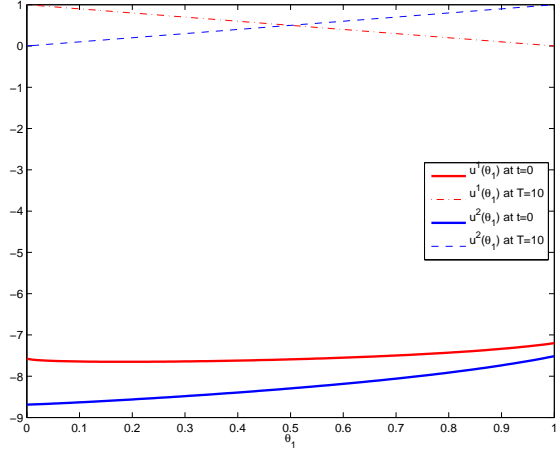}
\caption{Example II - paradigm shift}\label{f:ex2}
\end{center}
\end{figure} 
In figure \ref{f:ex2} we observe the paradigm shift within the scientific community. At $T=10$ the optimal states are $\theta_1=1$ and $\theta_2 = 1$ since the functions
$u_1$ and $u_2$ take their minimum value at these points respectively. In figure \ref{f:ex2} we observe that this is not the case at $t=0$. Here $u_1$ takes its minimum
value at $\theta_1 = 0$, i.e. paradigm 1 is not popular any more. 

\paragraph{Example III (Consumer choice):} 
In our final example we consider the consumer choice behaviour, see section \ref{ses}-(\ref{s:mobilephone}).
We set the final utility function to
\begin{align*}
u^1(\theta, 10) = 1-\theta_1 \text{ and } u^2(\theta, 10) = \theta_2.
\end{align*}
At time $T=10$ the utility functions take their minimum value at $\theta_1=1$ and $\theta_2=0$, i.e. their minimum corresponds to
the case that either all of them choose product 1 or product 2, respectively. Figure \ref{f:ex3} illustrates the utility functions for
two sets of parameters, namely
\begin{align*}
\eta = 0.5, ~ s_1 = 0.075, ~s_2 = 0.1 \text{ and } \eta = 1, ~s_1 = 0.1, ~ s_2 = 0.075.
\end{align*}
\begin{figure}[htb!]
\begin{center}
\subfigure[$\eta=0.5$, $s_1 = 0.075$ and $s_2 = 0.1$]{\includegraphics[width=0.475\textwidth]{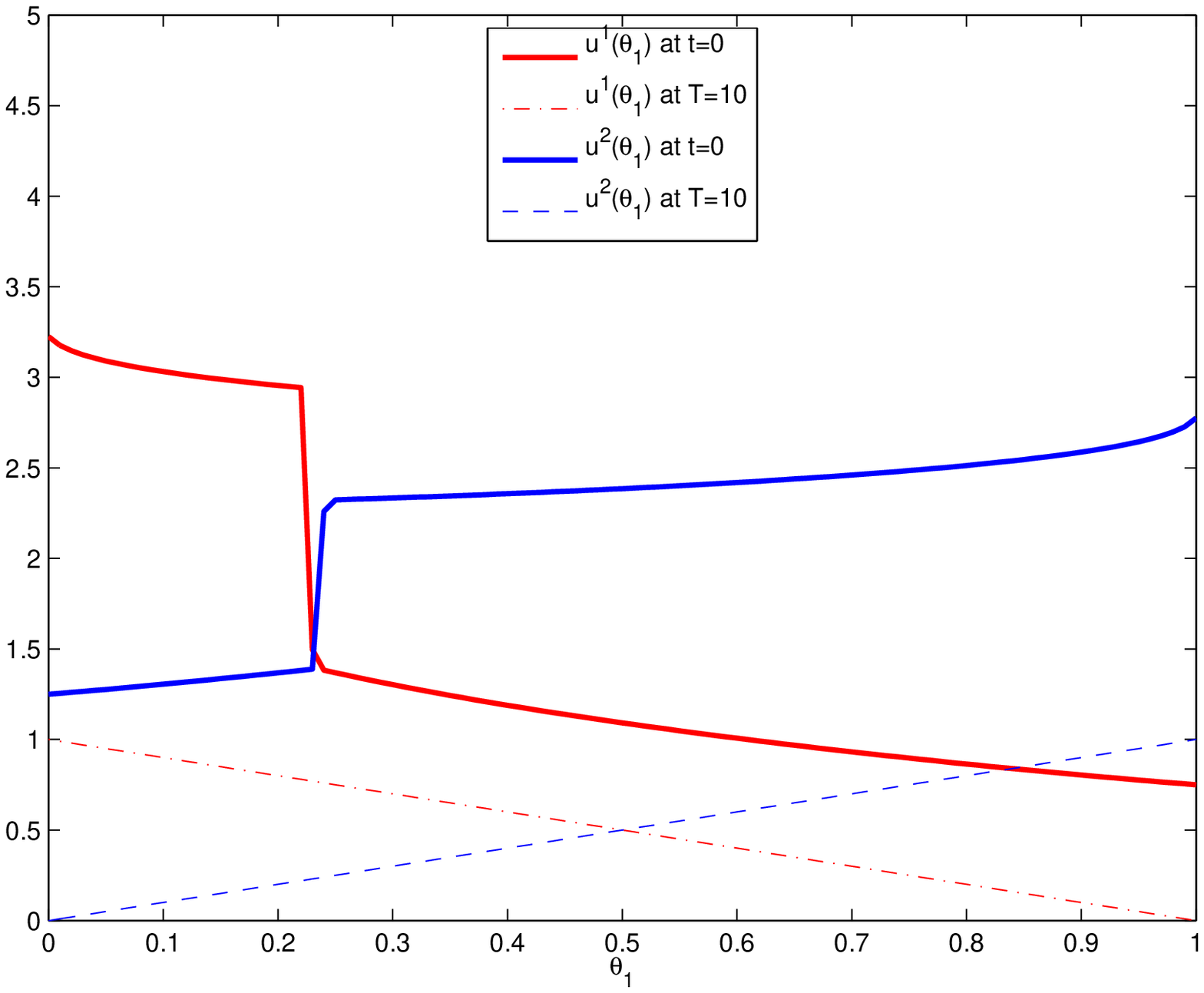} }
\subfigure[$\eta=1$, $s_1 = 0.1$ and $s_2 = 0.075$]{\includegraphics[width=0.475\textwidth]{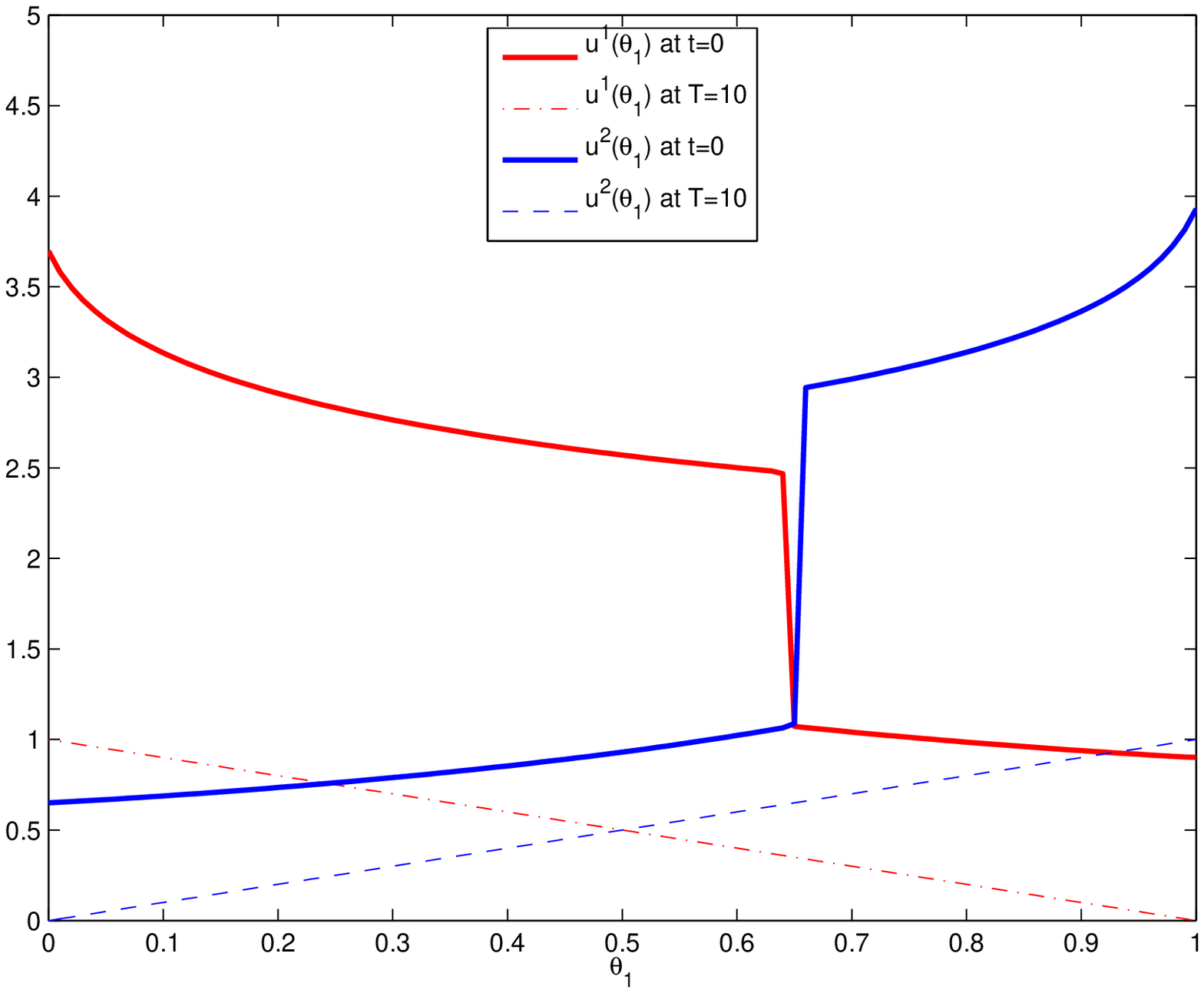} }
\caption{Example III - consumer choice }\label{f:ex3}
\end{center}
\end{figure}
We observe for the second set of parameters that $u^1$ takes its minimum value on the interval $\theta_1 \in [0.65,1]$. The jump at $\theta_1 \approx 0.65$ in both utilities
indicates the existence of a critical acceptance rate. If less than $65 \%$ of the consumers buy product 1, the price is increasing and the
product will not be competitive. 

\subsection{Potential mean field games}\label{s:numpotmfg}

\noindent In order to validate our methods, we consider Example~I, that can be written
as a potential mean field game  where $F$ and $H$ in \eqref{e:Hamiltonian} are given by
\begin{align*}
F(\theta_1, \theta_2) &= \theta_1 \theta_2, 
\end{align*}
\begin{align*}
H(p_1, p_2, \theta_1, \theta_2) &= - \frac{1}{2} \theta_1 ((p_1-p_2)^+)^2 - \frac{1}{2} \theta_2 ((p_2-p_1)^+)^2 + F(\theta_1, \theta_2).
\end{align*}
Then we compare the numerical simulations of \eqref{e:discrete}
with the ones for the corresponding Hamilton-Jacobi equation as we explain in what follows.  

\noindent For $i=1,2$, set
\begin{align*}
  f(i,\theta) = \frac{\partial}{\partial \theta_i} F(\theta_1, \theta_2) ~ \text{ and } ~ \Psi_0(\theta_1,\theta_2) = \frac{1}{2} \left(\theta_1 - \frac{1}{2} \right)^2 + \frac{1}{2} \left(\theta_2-\frac{1}{2}\right)^2.
\end{align*}
In order to simplify the numerical implementation, we perform a dimensionality reduction. 
Define $(\theta_1, \theta_2)=(\zeta, 1-\zeta)$, with $\zeta\in [0,1]$. 
 Let $\Psi$ be the solution to \eqref{e:hj} and define
\[
\Upsilon(\zeta, t)=\Psi(\zeta, 1-\zeta, t).
\]
We observe that 
\begin{equation}
\label{e:hj2}
-\frac{\partial \Upsilon}{\partial t}=\tilde H(\partial_\zeta \Upsilon, \zeta),
\end{equation}
where
\[
\tilde H(\partial_\zeta \Upsilon, \zeta)=- \frac{1}{2} \zeta \left[(\partial_\zeta \Upsilon)^+\right]^2 - \frac{1}{2} (1-\zeta) \left[(-\partial_\zeta \Upsilon)^+\right]^2 + F(\zeta, 1-\zeta).
\]

\noindent We use Godunov's method and an explicit Runge-Kutta method to discretize \eqref{e:hj2}.
Particular care has to be taken at the boundary. Since $\zeta \in [0,1]$ represents the first component of a probability vector, the natural boundary 
conditions for this problem are state constraints. A possibility to implement this is by 
supplementing \eqref{e:hj2} with large Dirichlet boundary values, i.e.
\begin{align}
\Upsilon(0,t) = \Upsilon(1,t) = c_D \text{ with } c_D \in \mathbb{R}^+.
\end{align}
\noindent To implement the Dirichlet boundary conditions we follow the works of Abgrall and Waagan, see \cite{Abgrall, Waagan}.
 Again we consider an equidistant discretization of the interval $[0,1]$  into $N$ subintervals of size $\Delta \zeta$, 
 and we approximate the solution $\Upsilon(\zeta, \tau)$ to \eqref{e:hj2} by $\Upsilon^\tau(\zeta) $
 for $\tau=T-l\Delta t $, and $\zeta=k \ \Delta\zeta$, $0\leq k\leq N$; $l,k\in \Nn_0$.  
 We set $\Upsilon^T(\zeta)=\Psi_0(\zeta, 1-\zeta)$. 
Then, the Godunov scheme can be written as
\begin{align}\label{e:godunov}
\Upsilon^{\tau-\Delta t} = \Upsilon^\tau - \Delta t  \ \hat{H} (\delta^-_\zeta \Upsilon^\tau,
 \delta^+_\zeta \Upsilon^\tau, \zeta),
\end{align}
where $\delta^+_\zeta$ and $\delta^-_\zeta$ are the difference operators
\begin{align*}
\delta^-_\zeta \Phi(\zeta) = \frac{\Phi(\zeta)-\Phi(\zeta-\Delta \zeta)}{\Delta \zeta} \text{ and }\delta^+_\zeta \Phi(\zeta) = \frac{\Phi(\zeta + \Delta \zeta)-\Phi(\zeta)}{\Delta \zeta}, 
\end{align*}
and $\hat{H}$ in \eqref{e:godunov} is given by
\begin{align*}
\hat{H}(\alpha, \beta, \zeta) &=
\begin{cases}
\min\limits_{\alpha \leq q \leq \beta} \tilde H(q,\zeta), \text{ if } \alpha \leq \beta; \\
\max\limits_{\beta \leq q \leq \alpha} \tilde H(q,\zeta), \text{ if } \beta \leq \alpha.
\end{cases}
\end{align*}
At the boundary $\zeta =0,1$, we set
\begin{align*}
\Upsilon^{\tau-\Delta t}(0) &= \min\left[\Upsilon^\tau(0) - \Delta t H^-(\delta^+_\zeta \Upsilon^\tau(0), 0),\  c_D\right], \\
\Upsilon^{\tau-\Delta t}(1) &= \min\left[\Upsilon^\tau(1) - \Delta t H^+(\delta^-_\zeta \Upsilon^\tau(1), 1),\  c_D\right],
\end{align*}
where $H^-(p,\zeta)  = \hat{H}(0,p,\zeta)$ and $H^+(p,\zeta) = \hat{H}(p,0,\zeta)$. Figure \ref{f:hj} shows the derivative of $U$ with respect to $\theta_1$ as well as the
difference $u^1-u^2$ calculated from \eqref{e:discrete} at time $t=0$. The same spatial and temporal discretization (i.e. $N=100$ and $\Delta t=10^{-4}$) was used in both simulations.
\begin{figure}
\begin{center}
\includegraphics[width=0.6\textwidth]{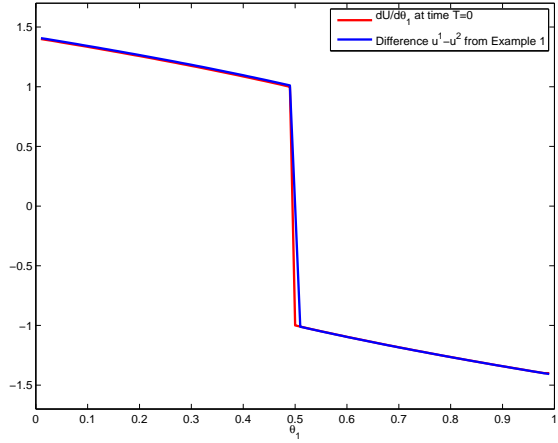}
\caption{Derivative of $U$ with respect to $\theta_1$ versus the difference $u^1-u^2$ calculated in Example~I.}\label{f:hj}
\end{center}
\end{figure}

\section{Shock structure for two-state problems}
\label{sstsp}

Finally we consider two-state problems and investigate in detail the shock structure. For this purpose we perform a reduction of the dimension (since the key issues are 
related to the difference of the values, more than to its proper value) and obtain a hyperbolic scalar equation. Then we introduce a related conservation law,
which yields a Rankine-Hugoniot condition for possible shocks. This new formulation allows us to study finite state mean field games from another numerical perspective and
 gives new insights into the shock structure and the qualitative behaviour of solutions to \eqref{hsys}. 

\subsection{Reduction to a scalar problem}
Let $U(\theta,t)$ be a $C^1$ solution to \eqref{hsys} with $d=2$. 
We define $w(\zeta, t)=U^1(\theta, t)~-~U^2(\theta, t),$ where $\theta=(\zeta, 1-\zeta)$.
From \eqref{hsys} we have that
\begin{align}
\begin{split}
\left( U^1 - U^2 \right)_t  &= -g_1(U^1,U^2,\theta_1,\theta_2)  \partial_{\theta_1} \left( U^1 - U^2 \right) 
 -g_2(U^1,U^2,\theta_1,\theta_2) \partial_{\theta_2} \left( U^1 - U^2 \right){}\\
&\phantom{=}- h(U^1,U^2,\theta_1,\theta_2,1) + h(U^1,U^2,\theta_1,\theta_2,2).\label{diffU}
\end{split}
\end{align}
Then $h$ and $g$, given by \eqref{hami} and \eqref{propg}, can be written as
$ h(U^1,U^2,\theta_1,\theta_2,i) = h(w(\zeta,t),0,\zeta,1-\zeta,i)$ 
and $g_1(U^1,U^2,\theta_1,\theta_2)= g_1(w(\zeta,t),0,\zeta,1-\zeta)$. 
For two-state problems equation \eqref{symmetrygj} gives $g_2 = - g_1$. Hence we obtain
\begin{align*}
w_t  = & - g_1(w(\zeta,t),0,\zeta,1-\zeta) \left[ \ \partial_{\theta_1} \left( U^1 - U^2 \right) - \partial_{\theta_2} \left( U^1 - U^2 \right) \right] \\
 	   &- h(w(\zeta,t),0,\zeta,1-\zeta,1) + h(w(\zeta,t),0,\zeta,1-\zeta,2).
\end{align*}
Define $r$ and $q$ by
\begin{align*}
r(w(\zeta,t),\zeta) &= - g_1(w(\zeta,t),0,\zeta,1-\zeta),\\
q(w(\zeta,t),\zeta) &= h(w(\zeta,t),0,\zeta,1-\zeta,1) - h(w(\zeta,t),0,\zeta,1-\zeta,2),
\end{align*}
and denote $\frac{\partial w}{\partial \zeta}$ by $ \frac{\partial w(\zeta,t)}{\partial \zeta} =
 \left( \frac{\partial}{\partial \theta_1} - \frac{\partial}{\partial \theta_2} \right)
 \left( U^1 - U^2  \right) |_{(\zeta,1-\zeta)}$.
Then equation \eqref{diffU} for the difference of $U^i$ can be written as
\begin{equation}
\label{scalar}
- w_t(\zeta,t) + r(w(\zeta,t),\zeta) \ \partial_{\zeta} w(\zeta,t) = q(w(\zeta,t),\zeta).
\end{equation}

\subsection{A conservation law and the Rankine-Hugoniot condition}

Consider the following conservation law associated to \eqref{scalar} on the interval~$[0,1]$,
\begin{equation}
\label{eqP}
- P_t(\zeta,t)  + \partial_{\zeta} \left( r(w(\zeta,t),\zeta) \ P(\zeta,t) \right) = 0, 
\end{equation}
supplemented with the boundary condition $P(0, t)=P(1,t)=0$ for all times $t \in [0,T]$.\\
\noindent If $P$ is a sufficiently smooth solution to \eqref{eqP} and $P(\zeta, 0)\geq 0$, then the maximum principle implies that $P(\zeta, t)\geq 0$. 
Furthermore, if $ \int P(\zeta, 0)d\zeta=1$, we have that $ \int P(\zeta, t) d\zeta=1$. Assuming
that $P(\zeta, 0)$ is a probability distribution, we can regard $P(\zeta,t)$ as a probability distribution on the
set $\Pp(\Ii)$, $\Ii=\{1,2\}$, as we have a natural identification for two-state problems: $\Pp(\Ii)\simeq [0,1]$. 
Hence equation \eqref{eqP} describes an evolution of a probability distribution on $\Pp(\Ii)$.
Uncertainty in the initial distribution of the mean field $\zeta$ can be encoded in the initial condition $P(\zeta,0)$ and propagated through \eqref{eqP}.

\noindent Since equation \eqref{eqP} may not have globally smooth solutions, we use the Rankine-Hugoniot condition to characterise certain possibly discontinuous
solutions. Let $s:[0,T]\to [0,1]$ be a $C^1$ curve and suppose that $P$ is a $C^1$ function on both 
$0<\zeta< s(t)$ and $s(t)<\zeta<1$, for $t\in [0,T]$. 
Assume further that \eqref{eqP} holds in this set. Let $B:[0,1]\times [0,T]\to \Rr$. 
We denote by $[B]$ the jump of $B$ across the curve $s(t)$, that is,
$
[B]=B(s^+(t),t)-B(s^-(t), t) 
$.
Equation \eqref{eqP} leads to the Rankine-Hugoniot condition of the form
\begin{equation} 
\label{RHcond}
[P] \dot s = - [ r \ P ]. 
\end{equation}
If we start with initial condition $P(\zeta, 0)=1$, then the support of $P$ is the closure of the set of all mean field states which can be reached from some initial mean field state (all possible choices of $\zeta$ at time $0$).
Suppose $\mathcal{B}=\{(\zeta, t)=(s(t), t),\ 0~\leq~t~\leq~T\}$. Suppose that there is a discontinuity in $P$ at the boundary $\mathcal{B}$ of the set $P=0$.  Then 
we conclude from the Rankine-Hugoniot condition that
\[
\dot s=-r. 
\]
Hence $\mathcal{B}$ is a characteristic for \eqref{scalar}.

\noindent
Using \eqref{eqP} we can also derive local Lipschitz bounds for the solution to \eqref{scalar}:
\begin{proposition}
Let $w$ be a smooth solution to \eqref{scalar}. Then there exists a time $t_0< T$ and a constant $C$,
depending only the $L^\infty$ norm and the Lipschitz constant of $w(x, T)$, such that 
for $t_0\leq t\leq T$, 
$$ \left\|\partial_\zeta w(\cdot, t) \right\|_{L^\infty([0,1])}\leq \frac C {(t-t_0)}.
$$  
\end{proposition}
\begin{proof}
Let $P$ be a solution to \eqref{eqP} with $P(0, t)=P(1,t)=0$, $t\leq T$, $P(\zeta, t)\geq 0$,
with $\int P(\zeta, t)d\zeta=1$. We differentiate equation \eqref{scalar} with respect to $\zeta$ and multiply it by $2 \partial_\zeta w$ and
obtain
\[
-\partial_t S+r\partial_\zeta S=-2\partial_\zeta r \partial_\zeta w-
2\partial_w r S \partial_\zeta w+2\partial_w q S+2\partial_\zeta q\partial_\zeta w,
\] 
using that $S(\zeta,t)=(\partial_\zeta w(\zeta,t))^2$. Then we can deduce the following estimate:
\begin{align*}
-\frac{d}{dt}\int S(\zeta,t) P(\zeta, t) \ d\zeta &\leq 2\int \left( -\partial_\zeta r \partial_\zeta w-
\partial_w r S \partial_\zeta w+\partial_w q S+\partial_\zeta q\partial_\zeta w \right) P (\zeta, t) \ d\zeta \\
& \leq C_1\int S^{3/2} P(\zeta, t)d\zeta+C_2, 
\end{align*}
where  we use the fact that $w$ is bounded, see \cite{GMS2}, in the last step. Then 
\[
\int S(\zeta,t) \ P(\zeta, t) \ d\zeta \leq C+\int_t^T \|S(\cdot, s)\|_\infty^{3/2}ds+\|S(\cdot, T)\|_{\infty}.
\]
Taking $P(\zeta,t)$ to be an arbitrary probability measure on $[0,1]$ this yields that 
\[
\|S(\cdot,t)\|_\infty \leq  C+\int_t^T \|S(\cdot, s)\|_\infty^{3/2}ds+\|S(\cdot, T)\|_{\infty}.
\]
This estimate does not give global bounds, but a nonlinear version of Gronwall's inequality yields the existence of $t_0$. 
\end{proof}

\subsection{Numerical analysis of the shock structure}

Finally we discuss the particular formulation of equations \eqref{scalar} and \eqref{eqP} as well as the numerical realisation of example~I presented in section \ref{s:numericalexamples}. Equation \eqref{scalar} can be written as 
 \begin{equation}\label{scalar:applied3}
 - w_t(\zeta,t) = \frac{(1-2\zeta)|w| - w}{2} \ \partial_\zeta w(\zeta,t) + \frac{1}{2} \ |w| \ w   - \left[ f(1,\zeta,1-\zeta) - f(2,\zeta,1-\zeta) \right].
 \end{equation}
using that $h$ and $g$ are given by \eqref{e:h} and \eqref{formg1} respectively.
Similarly, \eqref{eqP} takes the form
\begin{equation}
\label{eqP2}
- P_t(\zeta,t) + \partial_\zeta \left[ -\frac{(1-2\zeta) \ |w| - w}{2} P(\zeta,t) \right] = 0. 
\end{equation}
Note that the difference $w = u^1-u^2$ has to satisfy $w(0,t)\geq 0$ and $w(1,t)\leq 0$. The violation of this assertion corresponds to the
case of no agents being in state $1$, 
i.e. $\zeta=0$ in a situation where state $1$ would be preferred to state $2$
(and analogously for $\zeta=1$).
Therefore equation  \eqref{eqP2} does not require any boundary conditions since for $\zeta=0$ and $\zeta=1$ the advection 
 term vanishes, i.e.,
 $
 \frac{(1-2\zeta) \ |w| - w}{2}=0.
 $

\noindent We simulate \eqref{eqP2} using a semidiscrete central upwind scheme introduced by Kurganov et al., see \cite{Kurganov2001}, and
use the same parameters as in Section \ref{s:numericalexamples}. The initial datum $P(\zeta,0)$ is set to
\begin{align*}
P(\zeta,0) = 1 \text{ for } \zeta \in (0,1) \text{ and } P(\zeta,0) = 0 \text{ for } \zeta \in \lbrace 0,1 \rbrace.
\end{align*}
Due to the vanishing advection terms at $\zeta = 0,1$, we set homogeneous Dirichlet boundary conditions for $P$, 
that is $P(\zeta, t)=0$ if $\zeta \in \lbrace 0,1 \rbrace$. We choose an
equidistant discretization of $N=125$ points and times steps $\Delta t=10^{-4}$. The evolution of $P$ for example~I presented in section \ref{ses}  is illustrated in figure \ref{f:p}.
\begin{figure}[hftb!]
\begin{center}
\includegraphics[width=0.5\textwidth]{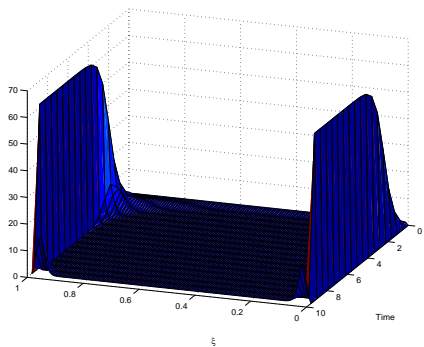} 
\caption{Evolution of $P$ for Example I discussed in Section \ref{ses}. }\label{f:p}
\end{center}
\end{figure}
We observe that the function $P$ does not see the shock in $w$ (see Figure~1), which is located at $\zeta=0.5$. 

\section{Conclusions}

In this paper we present effective numerical methods for two-state mean field games
and discuss a number of illustrative examples in socio-economic sciences. We compare our method with
numerical results obtained from classical and well established schemes 
for Hamilton-Jacobi equations. As presented in the examples, our method  captures shocks effectively. 
We analyse the shock structure using an associated conservation law and prove a local Lipschitz estimate 
for the solutions of \eqref{hsys}. 

\noindent Several challenging and interesting mathematical questions will be addressed in a future
paper, like the development of numerical schemes based on the dual variable method and the Lions transversal variable method,
see \cite{LCDF}.

\bibliographystyle{alpha}
\bibliography{mfg_urls}

\section*{Acknowledgements}
DG was partially supported by CAMGSD-LARSys (FCT-Portugal), PTDC/MAT-CAL/0749/2012 (FCT-Portugal), and KAUST SRI, Uncertainty Quantification Center in Computational Science and Engineering. RMV was partially supported by CNPq - Brazil through a PhD scholarship - Program Science without Borders. MTW acknowledges support from the Austrian Academy of Sciences \"OAW via the New Frontiers Project NST-001.

\end{document}